\documentclass[a4paper,UKenglish]{amsart}
 
\usepackage[inner=30mm, outer=30mm, textheight=245mm]{geometry} 
\usepackage{microtype}%if unwanted, comment out or use option "draft"

\usepackage[norelsize]{algorithm2e}
\providecommand{\SetLine}{\SetAlgoLined} % to take care of old command
\RestyleAlgo{boxruled}

\usepackage{tikz}
\usetikzlibrary{arrows,decorations.markings}
\usetikzlibrary{arrows,shapes,snakes,automata,backgrounds,petri}
\usetikzlibrary{trees}
\usetikzlibrary{snakes}
\usepackage{ellipsis}
\usepackage{ifthen}
\tikzset{middlearrow/.style={
        decoration={markings,
            mark= at position 0.5 with {\arrow{#1}} ,
        },
        postaction={decorate}
    }
}

\newcommand{\ov}[1]{\overline{#1}}

\newcommand{\invol}{\overline{\,^{\,^{\,}}}}
\newcommand{\ra}{\rightarrow}

\newcommand{\Ralg}{\ensuremath{\mathsf{REIN}}}

\theoremstyle{plain}
\newtheorem{thm}{Theorem}[section]
\newtheorem{lem}[thm]{Lemma}

\newtheorem{prop}[thm]{Proposition}
\theoremstyle{definition}

\newtheorem{example}[thm]{Example}

\title{2-manifold recognition is in logspace%
    \footnote{Supported by the Australian Research Council
    (project DP110101104).}}

\title{2-manifold recognition is in logspace}
\author[B.\,A. Burton]{Benjamin A.\ Burton}
\address{School of Mathematics and Physics, The University of Queensland, Brisbane QLD
4072, Australia}
\email{bab@maths.uq.edu.au}
\author[M. Elder]{Murray Elder}
  \address{School of Mathematical and Physical Sciences, The University of Newcastle, Callaghan NSW 2308, Australia}
\email{Murray.Elder@newcastle.edu.au}
\author[A. Kalka ]{Arkadius Kalka}
\address{Department of Mathematics, Bar Ilan University, Ramat Gan 52900, Israel}
\email{Arkadius.Kalka@rub.de}
\author[S. Tillmann]{Stephan Tillmann}
\address{School of Mathematics and Statistics, The University of Sydney, Sydney NSW 2006,
Australia}
\email{tillmann@maths.usyd.edu.au}
\thanks{Research supported by the Australian Research Council (project DP110101104)}
\date{\today} 

\begin{document}

\maketitle

\begin{abstract}
We prove that the homeomorphism problem for 2--manifolds can be decided in logspace.
The proof relies on Reingold's logspace solution to the undirected $s,t$-connectivity problem in graphs.
\end{abstract}

%%%%%%%%%%%%%%%%%%%%%%%%%%%%

\section{Introduction}

Two compact, connected surfaces with (possibly empty) boundary are
homeomorphic precisely when they have the same Euler characteristic, the
same number of boundary components and they are either both orientable
or both non-orientable. This triple of invariants leads to the
classification theorem for compact surfaces, which has roots in work of
Camille Jordan and August M\"obius in the 1860s, and Max Dehn and Poul
Heegaard in 1907, with the first rigorous proof due to Henry
Brahana~\cite{brahana-1921} in 1921 under the hypothesis that the
compact surfaces are triangulated; a modern proof with this hypothesis
is due to John Conway and presented by Francis and
Weeks~\cite{francis-ZIP}. The fact that every compact surface has a
triangulation, thus completing the classification theorem, was
established by Tibor Rad\'o~\cite{rado-triangulation} in 1925, and a
modern proof using the ``Kirby torus trick'' was recently given by Allen
Hatcher~\cite{hatcher-surfaces}.

Once a solution to a decision problem has been found, it is natural to
investigate its implementation and complexity. Lower dimensions are
often used to get a foothold into higher dimensions---for instance, many
algorithms for 3--dimensional manifolds follow Kneser's blueprint and
study them via surfaces embedded in them; the enumeration of
triangulated 4--manifolds requires us to recognise the 3--dimensional
sphere efficiently.

The main result of this paper asserts that given two finite
2--dimensional triangulations one can decide whether they
represent homeomorphic surfaces using space logarithmic in the size of
the triangulations. Letting $L$ be the class of problems that can be decided
in deterministic logspace, this is formally stated as:

\begin{thm} %\hspace{-0.15cm}{\bf .}
\label{Main}
{\sc 2-manifold Recognition} is in {\rm L}.
\end{thm}

The definitions of ``{\sc 2-manifold Recognition}'' and
``\emph{deterministic logspace}'' are given in \S\ref{sec:part1}. We
remark that while \emph{a priori} a logspace algorithm has no time
bound, it is easy to show that logspace algorithms run in polynomial
time (see for example Lemma 4 in \cite{EEO}).
%An important open problem is  whether {\rm L=P}.

There is a remarkable gap between surfaces and higher dimensional
manifolds. Manifolds of dimension three were shown to be triangulated by
Moise~\cite{moise} in 1952, and an excellent discussion of the current
state of the solution of the homeomorphism problem can be found in the
recent survey by Matthias Aschenbrenner, Stefan Friedl and Henry
Wilton~\cite{AFW-survey}. In a nutshell, the homeomorphism problem for
oriented 3--manifolds has been solved modulo the existence of an
algorithm that determines whether an oriented 3--manifold has an
orientation reversing involution. Many important algorithms for
3--manifolds have been implemented~\cite{Regina}, and many important
decision problems, such as unknot recognition~\cite{HLP} and 3--sphere
recognition~\cite{schleimer}, have been shown to be in the complexity
class {\rm NP}.

The next qualitative gap arises between dimensions three and higher.
There are compact 4--dimensional manifolds, such as the $E_8$ manifold
discovered by Mike Freedman in 1982, that are not homeomorphic to any
simplicial complex. Ciprian Manolescu~\cite{manolescu} has recently
announced that there are also such examples in dimensions $n\ge 5.$ Even
if one restricts to compact, simplicial manifolds of dimension $n\ge 4,$
the homeomorphism problem was shown to be undecidable by
Markov~\cite{markov} in 1958 as a consequence of the unsolvability of
the isomorphism problem for finitely presented groups, which is due to
Adyan~\cite{Ad57b, Ad57a} and Rabin\cite{Rab58}. In particular, for the
development of algorithms in higher dimensions one needs to restrict to
special classes of manifolds, or else be content with heuristic
methods.\\

We now give an informal description of our logspace algorithms.
The starting point is a logspace algorithm which, given a single triangulation
as input:
\begin{enumerate}\item 
Checks that the triangulation is a 2-manifold.
\item Counts the number of connected components, $c.$
\item If the input is a connected 2-manifold:
\begin{itemize}
\item decides if it is orientable or non-orientable;
\item computes the Euler characteristic, $\chi$;
\item counts the number of boundary components, $b$.
\end{itemize}
\item If there is more than one connected component, the algorithm outputs
the following data: $(o_1,\chi_1,b_1), \dots, (o_c,\chi_c,b_c) $ where
$o_i=0$ if the $i$-th connected component is  orientable and 1 otherwise,
$\chi_i$ is its Euler characteristic, and  $b_i$ is its number of
boundary components.
Moreover, the algorithm outputs this data in the following order:
\begin{itemize}
\item $o_i<o_{i+1}$, or 
\item $o_i=o_{i+1}$ and $\chi_i< \chi_{i+1}$, or
\item $o_i=o_{i+1},\chi_i =\chi_{i+1}$, and $b_i\leq b_{i+1}$.
\end{itemize}
\end{enumerate}

This output is a complete invariant of the homeomorphism type of the
2-manifold, and so the solution to the homeomorphism problem then follows
by running this algorithm simultaneously on two triangulations.

This paper is organised as follows. We give precise definitions of the
complexity class and data structures we use in \S\ref{sec:part1}. The
algorithms to verify that an input triangulation represents a surface and
count the number of components are described in
\S\ref{sec:algpartone}. The algorithm to compute the complete invariants
of a connected, triangulated surface is given in \S\ref{sec:part2},
and this algorithm is then applied in \S\ref{sec:moreconn} to compute the
invariants of each connected component of a disconnected surface.

%%%%%%%%%%%%%%%%%%%%%%%%%%%%%%%%

\section{Preliminaries}\label{sec:part1}
%\subsection{Logspace}
A {\em deterministic logspace transducer} consists of a finite state
control, a read-head, and three tapes: 
\begin{enumerate}
\item the {\em input tape} is read-only, and
stores the input string; 
\item the {\em work tape} is read-write, but is
restricted to using at most $c\log n$ squares, where $n$ is the length
of the word on the input tape and $c$ is a fixed constant; and 
\item the {\em output tape} is write-only, and is restricted to writing left
to right only. The space used on the output tape is not added to the space
bounds.
\end{enumerate}

A transition of the  transducer takes as input a letter of the input
tape at the position of the read-head, a state of the finite state
control, and a letter on the work-tape. On each transition the transducer
can modify the work tape, change states, and write at most a fixed
constant number of letters to the output tape, moving to the right along the
output tape for each letter printed.

Since the position of the read-head of the input tape is an integer
between 1 and $n$ (the length of the input), we can store it in  binary
on the work tape. In addition we can store a finite number of {\em
pointers} to positions on the input tape.

A problem is in deterministic logspace if it can be decided using a
deterministic logspace transducer.  Since all transducers in this
article will be determistic, we will say logspace for deterministic
logspace throughout.

A key property of logspace transducers is that they can be composed together
to give new logspace transducers. Formally, 
let $X,Y$ be finite alphabets, and let $X^*$ denote the set of all finite
length strings in the letters of $X$. %, including the empty string $\lambda$.
We call $f:X^*\rightarrow Y^*$ a {\em logspace computable function}  
if there is a   logspace transducer that on input $w\in X^*$ computes $f(w)$. 
\begin{lem}[Lemma 2 in \cite{EEO}]
\label{lem:comp}
If $f, g: X^* \rightarrow X^*$ can both be computed in logspace,
then their composition $f \circ g: X^* \rightarrow X^*$ can also be
computed in logspace.
\end{lem}
\begin{proof} %See Lemma 2 \cite{EEO}.
Let $M_f, M_g$ be logspace transducers that compute $f$ and $g$ respectively.
On input $w\in X^*$,
run $M_f$. Each time $M_f$ calls for the $j$th input letter, run $M_g$ on $w$;
however, instead of writing the output of $M_g$ to a tape, we add 1 to a
counter (in binary) each time $M_g$ would normally write a letter.
Continue running $M_g$ until the counter has value
$j-1$, at which point we return the next letter that $M_g$ would output
back to $M_f$.
\end{proof}

Finally, a \emph{logspace algorithm} is an algorithm that runs on a
logspace transducer.  Lemma~\ref{lem:comp} implies that a logspace
algorithm may assume that its input is the output of some other logspace
algorithm.

\bigskip

%\subsection{Triangulations}
In this article we show that the following decision problem is in logspace:\\

\begin{tabular}{l l}
\emph{Problem:} & {\sc 2-manifold Recognition}\\
\emph{Instance:} & Two 2--dimensional triangulations
$\mathcal T_1,\mathcal T_2$\\
\emph{Question:} & Do $\mathcal T_1$ and $\mathcal T_2$ represent homeomorphic 2--manifolds?
\end{tabular}\\

For a positive integer $n$ let $[n]$ denote the set $\{1, \ldots ,n\}$.
A triangulation $\mathcal T$ is specified by  a list of $n$ triangles,
where each triangle $t\in[n]$  has vertices labeled $1,2,3$ (which
induces an orientation on the triangle), and edges glued according to a
table  as follows:
\[
\begin{array}{c|c|c|c}
& (12) & (23) & (31)\\
\hline 
1 & a_1& b_1 & c_1\\
2 & a_2 & b_2 & c_2\\
\vdots &&&\\
n & a_n &  b_n &  c_n\\
\end{array}\]
with
\[a_t, b_t, c_t\in \{\emptyset\}  \cup  \{(s,e) \mid s\in[n],e\in \{(12),(21),(23),(32), (31),(13)\}\}.\]
The entry $a_t=(s,e)$ in row $t$ column $(12)$ means that the edge
$(12)$ in $t$ is glued to the edge $e$ in triangle $s$, whereas
$a_t=\emptyset$ means that the edge $(12)$ in $t$ is not glued to anything
(i.e., it is a boundary edge);
likewise for columns $(23)$ and $(31)$.

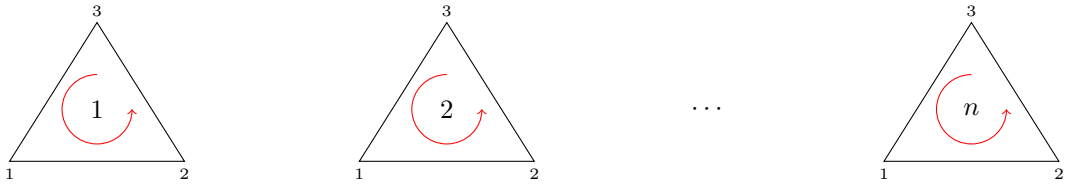
\begin{figure}[h]
\begin{center}
\begin{tikzpicture}[scale=2.3]

\draw  (0,0) -- (1,0) -- (.5,.8) -- cycle;
\draw  (2,0) -- (3,0) -- (2.5,.8) -- cycle;
\draw  (5,0) -- (6,0) -- (5.5,.8) -- cycle;

\tiny
\node[align=left, below] at (0,0) {$1$};
\node[align=left, below] at (2,0) {$1$};
\node[align=left, below] at (5,0) {$1$};
\node[align=right, below] at (1,0) {$2$};
\node[align=right, below] at (3,0) {$2$};
\node[align=right, below] at (6,0) {$2$};
\node[align=center, above] at (.5,.8) {$3$};
\node[align=center, above] at (2.5,.8) {$3$};
\node[align=center, above] at (5.5,.8) {$3$};

\normalsize

\node[align=center] at (.5,.3) {$1$};
\node[align=center] at (2.5,.3) {$2$};
\node[align=center] at (5.5,.3) {$n$};

\node[align=center] at (4,.3) {$\cdots$};

 \draw[red, ->] (.5,.5)   arc (90:360:2mm);
 \draw[red, ->] (2.5,.5)   arc (90:360:2mm);
  \draw[red, ->] (5.5,.5)   arc (90:360:2mm);
\end{tikzpicture}
\caption{Input triangles}
\label{fig:}\end{center}
\end{figure}

A triangulation is given to a logspace transducer by writing the string
\[ \# \ \  a_1\  \ b_1  \  \ c_1 \ \ \# \  \ a_2  \  \ b_2 \  \ c_2  \ \ \# \ \  \dots  \  \ \# \ \  a_n  \  \ b_n  \  \ c_n \] 
on the input tape 
using the alphabet $\{\#,  \emptyset, 0,1,   (12),(23),(31),(21),(32),(13)\}$  
where $a_i,b_i,c_i$ are written as either $\emptyset$ or a binary number followed by $(12),(23)$ or $(31)$.  

\begin{example}\label{eg:Klein}
Figure~\ref{fig:K triangulation} illustrates a triangulation of a
punctured Klein bottle, and Table~\ref{table:klein} shows the
corresponding table of edge gluings.
For this triangulation, the input tape of our logspace transducer would
read as follows:
\[\# \ 10 \ (13)  \  11 \ (12) \  11 \ (32) \ 
\# \ 11 \ (13) \  \emptyset \ 1 \ (21) \ 
\# \ 1 \ (23) \ 1 \  (13) \  10 \ (21) 
\]

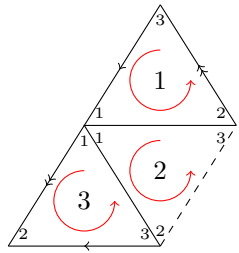
\begin{figure}[h]\label{fig:klein}
\begin{center}
\begin{tikzpicture}[scale=2]

\draw[middlearrow={<}] (0,0) --  (.5,.8);
\draw[middlearrow={>}] (.5,-.8) -- (-.5,-.8);

\draw[middlearrow={>>}] (0,0) --  (-.5,-.8);
\draw[middlearrow={>>}] (1,0) -- (.5,.8);

\draw  (1,0) -- (0,0) -- (.5,-.8);

\draw[dashed] (.5,-.8) -- (1,0);  

\node[align=center] at (.5,.3) {$1$};
\node[align=center] at (.5,-.3)  {$2$};
\node[align=center] at (0,-.5)  {$3$};

\tiny
\node[align=right, above] at (.1,0) {$1$};
\node[align=left, above] at (.9,0) {$2$};
\node[align=center] at (.5,.7) {$3$};

\node[align=right, below] at (.1,0) {$1$};
\node[align=center] at (.5,-.7) {$2$};
\node[align=left, below] at (.9,0) {$3$};

\node[align=center] at (0,-.1) {$1$};
\node[align=right, above] at (-.4,-.8) {$2$};
\node[align=left, above] at (.4,-.8) {$3$};

 \draw[red, ->] (.5,.5)   arc (90:360:2mm);
 \draw[red, ->] (0,-.3)   arc (90:360:2mm);
 \draw[red, ->] (.5,-.1)   arc (90:360:2mm); 

\end{tikzpicture}
\caption{Triangles for Example~\ref{eg:Klein} with edges identified. Dashed line is the boundary.}
\label{fig:K triangulation}\end{center}
\end{figure}

\begin{table}[h]\caption{Input table for Example~\ref{eg:Klein}}\label{table:klein}
$\begin{array}{c|c|c|c}
& (12) & (23) & (31)\\
\hline 
1 & 2,(13) & 3,(12) & 3,(32)\\
2 & 3,(13) & \emptyset & 1,(21) \\
3 & 1,(23) & 1,(13) & 2,(21)\\
\end{array}$
\end{table}

\end{example}

A triangulation of $n$ triangles has input size $N\in O(n\log n)$.  We
will prove that the data required to identify the homeomorphism type of
the input can be output by a  transducer using  $O(\log N)$ squares of
the work tape. It follows that on input a triangulation with $n$
triangles, the homeomorphism type can be computed using $O(\log
N)=O(\log n+\log\log n)=O(\log n)$ space. 

It can easily be checked in logspace that the input is written in the
required form -- the number $n$ of $\#$ symbols can be computed and
written in binary on the work tape by scanning  $\#$ symbols, then one
can check that  each binary number on the tape has value  between $1$
and $n$.  So we may assume the input is correctly specified. However, we
do not assume  that the gluing instructions are consistent or give a
manifold. For example, we may have $\emptyset$ in row $t$ column $(12)$
but $(t,(12))$ may appear as a different entry in the table, which would be
inconsistent.  The algorithm we describe will check this.

Note that the set $\epsilon=\{(12),(21),(23),(32), (31),(13)\}$ comes natually
equipped with an involution $\invol:\epsilon\ra \epsilon$
given by $\ov {(ij)}=(ji)$.

When describing our algorithms we will refer to {\em row $t$ column $e$}
of the input tape, which means the entry in row $t$ column $e$ of the
gluing table.  This can be located in logspace by  scaning the input
tape from left to right counting the number of  $\#$ symbols.

Throughout this paper we make use of a deterministic logspace algorithm
due to Reingold \cite{Rein} which takes input $(V,E,s,t),$ where $(V,E)$
is an undirected graph,  $s,t\in V$, and returns {\em Yes} if there is
an edge path from $s$ to $t$, and {\em No} otherwise.  We call this
algorithm $\Ralg$.

We present the algorithms in this paper using pseudocode.
Note that for-loops in the pseudocode are straightforward to implement
a logspace transducer, using a binary number on the work tape for each loop.
All of our algorithms make implicit use of the fact that
logspace functions are closed under composition
(Lemma~\ref{lem:comp}).

%%%%%%%%%%%%%

\section{Initial tests}\label{sec:algpartone}

\subsection{Counting components of a graph}

We begin with a simple tool that we call upon repeatedly in this paper:
a logspace algorithm to compute the number of connected components of an
undirected graph.
This algorithm follows immediately from $\Ralg$.  It operates as
follows; see Algorithm~\ref{alg:connected} for the pseudocode.

Assume the graph is
written on the input tape with vertices $[n]$ and edges given as a list
$E\subseteq [n]\times [n]$.
Initialise a counter $c=1$ for the component containing vertex~1.
The key idea is to iterate through the remaining vertices, and to increment $c$
each time we encounter the lowest-numbered vertex of some connected component.

More precisely:
Algorithm~\ref{alg:connected} runs through each vertex $t>1$,
calling $\Ralg$ to test whether $t$ is connected to any vertex $s<t$. If
it is, we leave $c$ unchanged and move to the next vertex. If it is
not, we increment $c$ and move to the next vertex.

\begin{algorithm}[h]
\caption{Count connected components.}
\label{alg:connected}
\SetLine
\KwIn{Undirected graph $([n], E)$.}
\KwOut{Number of connected components, $c$.}
Write a counter $c=1$ (in binary) to the work tape\;
\For{$t=2 \,\, {\tt to} \,\, n$}{
   Set $b={\tt false}$\;
   \For{$s=1 \,\, {\tt to} \,\, t-1$}{
        Run $\Ralg$ on  $([n], E, s,t)$. 
        If $\Ralg$ returns ${\tt true}$, set $b={\tt true}$\;
   }
   If  $b={\tt false}$, increment $c$ by 1\;
}
\Return{$c$\;}
\end{algorithm}

\subsection{Checking the input is a surface}

Our first task is to test the validity of the input.
Algorithm~\ref{alg:checksurface} decides whether the input represents a
surface by enumerating through  $(t,e)$ for
each $t\in[n],e\in \{(12),(23),(31)\}$ 
and checking: \begin{enumerate}
\item that the entry $(t,e)$ or $(t,\ov e)$ appears at most once in the table;
\item if row $t$ column $e$ of the table  is $\emptyset$, 
that neither $(t,e)$ nor $(t,\ov e)$  appear in the table;
\item if row $t$ column $e$ of the table  is $(s,f)$
for $f\in \{(12),(23),(31)\}$, that row $s$ column $f$ is $(t, e)$;
\item if row $t$ column $e$ of the table  is $(s,f)$
for $f\in \{(21),(32),(13)\}$, that row $s$ column $\ov f$ is $(t, \ov e)$;
\item that row $t$ column $e$ of the table is not $(t,e)$ or $(t,\ov e)$.
\end{enumerate}

\begin{algorithm}[h]
\caption{Check surface.}
\label{alg:checksurface}
\SetLine
\KwIn{Triangulation data on input tape.}
\KwOut{{\em Yes} if input is a surface, {\em No} otherwise.}
\For{$t\in [n]$}{
   \For{$e\in \{(12),(23),(31)\}$}{
        Write a counter $c=0$ to the work tape\;
        Scan the tape from left to right reading each entry $(s,f)$\;
        \Indp
        If $(s,f)\in\{(t,e),(t,\ov e)\}$, increment $c$ by 1\;
        If $c>1$,  output  {\em No} and stop\;
        \Indm
        Read the entry $y=(s,f)$ in row $t$ column $e$\;     
        If $y=\emptyset$ and $c\neq 0$, output  {\em No} and stop\;
        If $f\in \{(12),(23),(31)\}$, read the entry $z$ in row $s$ column $f$. If  $z \ne (t,e)$, output  {\em No} and stop\;
        If $f\in \{(21),(32),(13)\}$, read the entry $z$ in row $s$ column $\ov f$. If  $z \ne (t,\ov e)$, output  {\em No} and stop\;
        If $y\in \{(t,e),(t,\ov e)\}$, output  {\em No} and stop\;
   }
}
If {\em No} not printed,  return {\em Yes}\;
\end{algorithm}

\subsection{Counting the number of connected components}\label{subsec:connectedcomp}

Next, we count the number of connected components of the input surface.
To do this we construct the \emph{face-dual graph} of the surface,
which is an undirected graph whose vertices correspond to the triangles
of the surface,
and whose edges correspond to triangle gluings.
More precisely, the vertices of the face-dual graph are $[n]$, and the edges
of the face-dual graph are pairs
$(s,t)$ for which $(t,e)$ is in row $s$ of the table
for some $e\in \{(12),(23),(31),(21),(32),(13)\}$
(and therefore $(s,e)$ appears in row $t$ for some $e$ also).

Note that the face-dual graph as defined here is a simple graph:
it does not include loops or parallel edges
(which do not affect connectivity).
As an example, the face-dual graph for the punctured Klein bottle
of Example~\ref{eg:Klein} has vertices $\{1,2,3\}$, and edges
$\{1,2\}$, $\{2,3\}$, $\{1,3\}$.

Algorithm~\ref{alg:facedual} takes triangulation data as input
and outputs the face-dual graph as an undirected graph $([n],E')$,
using a simple scan through the table.

\begin{algorithm}[h]
\caption{Construct the face-dual graph.}
\label{alg:facedual}
\SetLine
\KwIn{Triangulation data on input tape.}
\KwOut{Face-dual graph $([n], E')$.}
Scan the tape counting $\#$ symbols in binary on the work tape, then store this number $n$ and write $[n]$ to the output tape\;
\For{$t\in[n]$}{
   \For{$s=t+1,\dots, n$}{
        Check whether $(t,e)$ is in row $s$ of the table
         for some $e\in \{(12),(21),(23),(32), (31),(13)\}$.
        If true, write  $(s,t)$ to the output tape\;
   }
}
\end{algorithm}

To count components of the input triangulation in logspace,
we use Algorithm~\ref{alg:facedual} to construct the face-dual graph,
and we count components using
Algorithm~\ref{alg:connected} with this face-dual graph as input.
Call the composition of these algorithms
\emph{Algorithm~A}.

%%%%%%%%%%%%%%%%%%%%%%%%%
%%%%%%%%%%%%%%%%%%%%%%%%%

\section{Algorithm for one connected component}\label{sec:part2}

In this section we assume the input surface is connected and compute its
homeomorphism type. In the next section we extend this to surfaces with
more than one connected component.

%%%%%%%%%%%%%%%%%%%%%%%%%

\subsection{Orientability}

We can determine whether or not a manifold is orientable by taking its
double cover, which is connected  if the manifold is non-orientable, and
which has two components if the manifold is orientable.

Recall that each
triangle has a fixed orientation determined by the corner labels
$1,2,3$.  The double cover is given by  a set of $2n$ triangles $\{t,t' \mid
t\in[n]\}$ with a gluing table constructed from the original table as
follows:
\begin{enumerate}
\item the rows of the table are  $\{t,t' \mid t\in[n]\}$ and the
columns  are $\{(12),(23),(31)\}$;
\item if row $t$ column $e$ of the original table contains $\emptyset$,
then write $\emptyset$ in rows $t,t'$ column $e$ of the new table;
\item if row $t$ column $e$ of the original table contains $(s,f)$ with
$f\in \{(12),(23),(31)\}$, then write $(s',f)$ in row $t$ column $e$ and
$(s,f)$ in row $t'$ column $e$  of the new table;
\item if row $t$ column $e$ of the original table contains $(s,f)$ with
$f\in \{(21),(32),(13)\}$, then write $(s,f)$ in row $t$ column $e$ and
$(s',f)$ in row $t'$ column $e$  of the new table.
\end{enumerate}
Table~\ref{table:doublecoverKlein} illustrates the double cover of the
punctured Klein bottle from Example~\ref{eg:Klein}.

\begin{table}[h]
\caption{Double cover data for Example~\ref{eg:Klein}}
\label{table:doublecoverKlein}
$\begin{array}{c|c|c|c}
&  (12)& (23) &(31) \\
\hline 
1 & 2,(13) & 3',(12) & 3,(32) \\
2 & 3,(13) & \emptyset & 1,(21) \\
3 & 1',(23) & 1,(13) & 2,(21)\\
1' & 2',(13) & 3,(12) & 3',(32) \\
2' & 3',(13) & \emptyset & 1',(21) \\
3' & 1,(23) & 1',(13) & 2',(21)\\
\end{array}$
\end{table}

We can easily describe a logspace algorithm to produce this double cover
gluing table from
the original input; see Algorithm~\ref{alg:double} for the details.
To test the orientability of the original input triangulation, we now
compose this with Algorithm~A from Section~\ref{subsec:connectedcomp}:
Algorithm~\ref{alg:double} constructs the double cover, and
Algorithm~A tests whether the double cover has one or two components.

\begin{algorithm}[h]
\caption{Construct the double cover of a triangulation.}
\label{alg:double}
\SetLine
\KwIn{Triangulation data on input tape.}
\KwOut{Triangulation data for the double cover, with triangles ordered as $1,\dots, n, 1',\dots, n'$.}
\For{$t\in [n]$}{
   Write $\#$ to the output tape\;
   \For{$e\in\{(12),(23),(31)\}$}{
       If row $t$ column $e$ of the input table is $\emptyset$, write $\emptyset $ to the output tape\;
        If row $t$ column $e$ of the input table is $(s,f)$ with   $f\in \{(12),(23),(31)\}$, write $s' \ f$ to the output tape\;
     If row $t$ column $e$ of the input table is $(s,f)$ with   $f\in \{(21),(32),(13)\}$, write $s \ f$ to the output tape\;
   }
}
\For{$t\in[n]$}{
   Write $\#$ to the output tape\;
   \For{$e\in\{(12),(23),(31)\}$}{
       If row $t$ column $e$ of the input table is $\emptyset$, write $\emptyset $ to the output tape\;
        If row $t$ column $e$ of the input table is $(s,f)$ with   $f\in \{(12),(23),(31)\}$, write $s \ f$ to the output tape\;
     If row $t$ column $e$ of the input table is $(s,f)$ with   $f\in \{(21),(32),(13)\}$, write $s' \ f$ to the output tape\;
   }
}
\end{algorithm}

%%%%%%%%%%%%%%%%%%%%%%%%%

\subsection{Euler characterisitic}\label{subsec:euler}

For a triangulation of a surface $S$ we have $\chi(S)=|V|-|E|+n$, where
$V$ and $E$ are the vertex set and edge set of $S$ respectively,
and where $n$ is the number of triangles.

Let $x$ be the number of edges of triangles that are not glued to any
other edge, i.e., the number of $\emptyset$ symbols on the input tape.
Then the number of edges is $|E| = (3n-x)/2+x=(3n+x)/2$,
since the remaining $3n-x$ triangle edges are identified in pairs.
We can compute $n$ and $x$, and hence $|E|$, in logspace by
counting the number of $\#$ and $\emptyset$ symbols on the input
tape.\footnote{Note that addition and division by two are both
logspace computable.}

It remains to compute $|V|$. We do this by tracking the identifications
of individual vertices of triangles. For this we construct an
undirected graph $K$,
which we call the \emph{vertex identification graph}, as follows.
The graph $K$ has vertex set
\[W=\{w_{t,1}, w_{t,2}, w_{t,3}\mid t\in[n]\},\]
where $w_{t,i}$ represents vertex~$i$ of triangle~$t$.
Note that the graph $K$ has $|W|=3n$ vertices overall.
In the punctured Klein bottle from Example~\ref{eg:Klein}, these
vertices are
\begin{align*}
W = \{ &
w_{1,1},\ 
w_{1,2},\ 
w_{1,3},\ 
w_{2,1},\ 
w_{2,2},\ 
w_{2,3},\ 
w_{3,1},\ 
w_{3,2},\ 
w_{3,3}\}.
\end{align*}

The edge set of the graph $K$ is
$F=\{(w_{t,i},w_{s,j}) \mid \mbox{$w_{t,i},w_{s,j}$ are identified
directly}\}$,
where by ``identified directly'' we mean that some edge triangle $t$
is glued to some edge of triangle $s$ in a way that maps vertex~$i$
of triangle~$t$ to vertex~$j$ of triangle~$s$.
For the punctured Klein bottle example, this edge set is
\begin{align*}
F = \{
    &\{w_{1,1}, w_{2,1}\},\ 
    \{w_{1,2}, w_{2,3}\},\\
    &\{w_{2,1}, w_{3,1}\},\ 
    \{w_{2,2}, w_{3,3}\},\\
    &\{w_{1,1}, w_{3,2}\},\ 
    \{w_{1,3}, w_{3,3}\},\\
    &\{w_{1,2}, w_{3,1}\},\ 
    \{w_{1,3}, w_{3,2}\}\}.
\end{align*}

Algorithm~\ref{alg:identificationgraph} constructs this graph in
logspace, simply by walking through the gluing table for the input
triangulation.  Note that, as it is presented here,
Algorithm~\ref{alg:identificationgraph} writes each edge to the output
tape twice; if desired this can easily be avoided using a lexicographical test.

\begin{algorithm}[h]
\caption{Construct the vertex identification graph $K=(W,F)$.}
\label{alg:identificationgraph}
\SetLine
\KwIn{Triangulation data on input tape.}
\KwOut{The graph $K=(W,F)$.}
\For{$t\in[n]$}{
      \For{$i\in\{1,2,3\}$}{
        Write $w_{t,i}$ to the output tape\;
	}
	}
\For{$t\in[n]$}{
   \For{$e=(ij)\in\{(12),(23),(31)\}$}{
      Read the entry $y=(s,(pq))$ in row $t$ column $e$\;
      If $y \ne \emptyset$, write $(w_{t,i}, w_{s,p})$ and $(w_{t,j}, w_{s,q})$ to the output tape\;
      }
   }
\end{algorithm}

Two vertices of $K$ are in the same connected component of $K$
if and only if the corresponding triangle vertices are identified in the
input triangulation, and so $|V|$ is the number of connected components
of the graph $K$.
Algorithm~\ref{alg:connected} with input $K=(W,F)$ computes this
number, and from this we can now compute the Euler characteristic
$\chi(S)$ of the input surface.

%%%%%%%%%%%%%%%%%%%%%%%%%

\subsection{Number of boundary components}

To count the number of boundary components in our surface,
we build another auxiliary graph $K'$, which we call the
\emph{boundary identification graph}.
This begins with the vertex identification graph $K$, and
introduces additional edges that join together different paths in $K$
that correspond to vertices on the same boundary component of the surface.

More precisely, this graph $K'$ has vertex set $W'=W$ as described
above.  The edge set of $K'$ is
\[ F'= F \cup \{(w_{t,i},w_{t,j}) \mid \mbox{edge $(ij)$ of triangle $t$
is not glued to anything}\}.\]
For the punctured Klein bottle example, this edge set is
\begin{align*}
F' = \{
    &\{w_{1,1}, w_{2,1}\},\ 
    \{w_{1,2}, w_{2,3}\},\\
    &\{w_{2,1}, w_{3,1}\},\ 
    \{w_{2,2}, w_{3,3}\},\\
    &\{w_{1,1}, w_{3,2}\},\ 
    \{w_{1,3}, w_{3,3}\},\\
    &\{w_{1,2}, w_{3,1}\},\ 
    \{w_{1,3}, w_{3,2}\},\\
    &\{w_{2,2}, w_{2,3}\}\}.
\end{align*}
Algorithm~\ref{alg:bdryidentificationgraph} shows how the graph $K'$ is
constructed.

\begin{algorithm}[h]
\caption{Construct the boundary identification graph $K'=(W',F')$.}
\label{alg:bdryidentificationgraph}
\SetLine
\KwIn{Triangulation data on input tape.}
\KwOut{The graph $K'=(W',F')$.}
\For{$t\in[n]$}{
      \For{$i\in\{1,2,3\}$}{
        Write $w_{t,i}$ to the output tape\;
	}
	}
\For{$t\in[n]$}{
   \For{$e=(ij)\in\{(12),(23),(31)\}$}{
      Read the entry $y=(s,(pq))$ in row $t$ column $e$\;
      If $y \ne \emptyset$, write $(w_{t,i}, w_{s,p})$ and $(w_{t,j}, w_{s,q})$ to the output tape\;
      If $y = \emptyset$, write $(w_{t,i}, w_{t,j})$ to the output tape\;
      }
   }
\end{algorithm}

We can analyse the structure of the vertex identification graph $K$
and the boundary identification graph $K'$:
\begin{itemize}
    \item $K$ is a disjoint union of cycles and paths,
    with one cycle for each internal vertex of the surface,
    and one path for each boundary vertex of the surface.
    \item $K'$ is a disjoint union of cycles,
    with one cycle for each internal vertex of the surface,
    and one cycle for each boundary component of the surface.
\end{itemize}
Moreover, the number of boundary vertices of the surface is equal to the
number of boundary edges; that is, the number of $\emptyset$ symbols in
the gluing table for the input triangulation.
Therefore counting boundary components becomes a simple matter of
counting boundary edges and counting components of $K$ and $K'$.
Algorithm~\ref{alg:boundary} gives the details.

\begin{algorithm}[h]
\caption{Count the number of boundary components.}
\label{alg:boundary}
\SetLine
\KwIn{Triangulation data on input tape.}
\KwOut{Number of boundary components, $b$.}
Compose Algorithms~\ref{alg:identificationgraph}
and~\ref{alg:connected} to find $k$, the number of
connected components of $K$\;
Compose Algorithms~\ref{alg:bdryidentificationgraph}
and~\ref{alg:connected} to find $k'$, the number of
connected components of $K'$\;
Count the number of $\emptyset$ symbols on the input tape,
and store this as the integer $x$\;
Write $b = k'-k+x$ to the output tape\;
\end{algorithm}

We summarise this section in the following statement.
\begin{prop}\label{prop:onecomponent}
There is a logspace algorithm, Algorithm~B, which given a
triangulation of a connected surface $S$ as input, computes $(o,\chi, b)$ where
$o=0$ if $S$  is orientable and $1$ if nonorientable, $\chi=\chi(S)$ is
the Euler characteristic of $S$,
and $b$ is the number of boundary components of $S$.
\end{prop}

%%%%%%%%%%%%%%%%%%%%%%%%%%%%

\section{More than one connected component}\label{sec:moreconn}

We now assume the output to Algorithm~\ref{alg:connected} is $c>1$.
We will compute the following data:
 \[(o_1,b_1,\chi_1), \dots, (o_c,b_c,\chi_c), \]
where  $o_i=0$ if the $i$-th connected component is  orientable and 1 otherwise,
$b_i$ is its number of boundary components, and
$\chi_i$ is its Euler characteristic.
Moreover, we output this data in the following order:
\begin{itemize}
\item $o_i<o_{i+1}$, or 
\item $o_i=o_{i+1}$ and $\chi_i< \chi_{i+1}$, or
\item $o_i=o_{i+1},\chi_i =\chi_{i+1}$, and $b_i\leq b_{i+1}$.
\end{itemize}

The pseudocode is shown below, and followed by a discussion of the meta-algorithm.

\begin{algorithm}
\caption{Outputs the triangulation data for the $i$-th connected component only.}
\label{alg:singlecomponent}
\SetLine
\KwIn{Triangulation data on input tape with $n$ triangles; integer $i\leq n$}
\KwOut{Triangulation data for the connected surface which is the $i$-th connected component of the input surface.}
Initialise counters $t=1$ and $c=1$ (in binary) on the work tape\;
\While{$c < i$}{
    Increment $t$ by 1\;
    Set $b={\tt false}$\;
    \For{$s=1 \,\, {\tt to} \,\, t-1$}{
         If $\Ralg([n], E, s, t)$ returns ${\tt true}$, set $b={\tt true}$\;
    }
    If  $b={\tt false}$, increment $c$ by 1\;
}
\For{$s=t \,\, {\tt to} \,\, n$}{
        \If{$\Ralg([n], E, s, t)$ returns ${\tt true}$}{
        Write $\#$ to the output tape\;
       \For{$e\in\{(12),(23),(31)\}$}{
          Read the entry $y=(u,f)$ in row $s$ column $e$\;
          \If{$y = \emptyset$}{
            Write $\emptyset$ to the output tape\;
          }\Else{
            Initialise counter $u'=0$\;
            \For{$x=t \,\, {\tt to} \,\, s$}{
              If $\Ralg([n],E,x,t)$ returns ${\tt true}$ then increment $u'$ by $1$\;
            }
            Write $(u',f)$ to the output tape\;
          }
          }
       }
}
\end{algorithm}

Here is the meta-algorithm. Assume we have checked that the input is a
surface, computed the number of triangles $n$ and counted the number
of connected components
$c>1$.  Then using the algorithms described above, we  compute the number
of   boundary components $b$ and Euler characteristic $\chi(S)$ for the
entire (disconnected) surface $S$.  Note that the number of  boundary
components (and connected components) is at most $n$. 

The Euler characteristic for a connected surface is at most 2.   We
compute a lower bound on $\chi$ for each connected component as follows.
If $S=\cup_{i=1}^c S_i$ are the
connected components, we have $\chi(S)=\sum_{i=1}^c \chi(S_i)$ so for
one component we have $\chi(S_{i_0})=\chi(S)-\sum_{i\neq i_0}
\chi(S_i)$. This is minimised when the negative term on the left is
maximised, and since the maximum Euler characteristic for any connected
surface is 2, we have $\chi(S_{i_0})\geq \chi(S)-2(c-1)$.

We then enumerate through all possible triples $(o, x,\chi)$ for
$o\in\{0,1\}, 0\leq x\leq b , \chi(S)-2(c-1)\leq \chi\leq 2$ (which is a
finite list), in the order given above. 
For each such triple $(o, \chi, x)$, we run through each connected
component of $S$ and compute $(o_i,\chi_i,b_i)$, and if
$(o_i,\chi_i,b_i)=(o,\chi,x)$ then we write this triple to the output.

This meta-algorithm repeatedly uses Algorithms~A and~B above.
It also requires a logspace algorithm for extracting the $i$-th connected
components of the input surface.  We present such a procedure in
Algorithm~\ref{alg:singlecomponent} which takes as input an integer $i$
and triangulation data for a surface with possibly many connected
components, and outputs the triangulation data for only the $i$-th
connected component.

Algorithm~\ref{alg:singlecomponent} is a straightforward extension of
Algorithm~\ref{alg:connected} (which just counts connected components).
We draw attention to the final loop over the counter $x$, which is used
to reindex the triangles on the output tape so that they are numbered
consecutively as $1,2,\ldots,k$, where $k$ is the number of triangles in
the $i$-th component.

%%%%%%%%%%%%%%%%%%%%%%%%%%%%

\section{Concluding remarks}

The main result of this paper is in fact stronger than presented in
the statement of Theorem~\ref{Main} (that 2-manifold recognition is in L):
the proof gives a logspace algorithm for the
\emph{function problem} to compute the homeomorphism type (essentially
a ``normal form'') of a given 2-manifold.
This is in contrast to problems on some groups, such as braid
groups with at least four strands, where there is a logspace solution
to the word problem \cite{MR1815219, MR1888796, MR0445901} but no known logspace algorithm  for computing
a normal form.

The only essential use of Reingold's
$s,t$ connectivity algorithm in our work is in counting components of the
surface, and testing orientability.
In particular, the auxiliary graphs $K$ and $K',$ which we use to
compute Euler characteristic and count boundary components, have a very
simple structure, and it is straightforward (but a little messier)
to design custom logspace algorithms for counting their components
that do not rely on $\Ralg$ as an oracle.
This raises the possibility that for connected and orientable
surfaces, homeomorphism testing may be even simpler---for instance,
an $\mathrm{NC}^1$ solution might be possible.

%Algorithm~\ref{alg:singlecomponent} outputs only those triangles that belong to the connected component containing triangle $p$, which is the $i$-th connected component of $S$.

%%%%%%%%%%%%%%%%%%%%%%%%%%%%%%%%
\bibliographystyle{plain}
\bibliography{refs}

\end{document}